\documentclass[12pt,reqno]{amsart}
\usepackage{amsmath, amsthm, amsopn, amssymb, microtype}
\usepackage{mathrsfs} 
\usepackage{mathscinet}
\usepackage{bbm}
\usepackage{enumerate, tikz, etoolbox, intcalc, geometry, caption, subcaption, afterpage}
\usepackage[english]{babel}

\title{Well-distribution of Polynomial maps on locally compact groups} 
\author{Tom Meyerovitch}
\address{Ben Gurion University of the Negev.
	Departement of Mathematics.
	Be'er Sheva, 8410501, Israel
}

\usepackage[colorlinks=true,urlcolor=blue,pdfborder={0 0 0}]{hyperref}
\hypersetup{linkcolor=[rgb]{0,0,0.6}}
\hypersetup{citecolor=[rgb]{0,0.6,0}}
\usepackage[nameinlink]{cleveref}
\crefname{theorem}{Theorem}{Theorems}
\crefname{thm}{Theorem}{Theorems}
\crefname{mainthm}{Theorem}{Theorems}
\crefname{lemma}{Lemma}{Lemmas}
\crefname{lem}{Lemma}{Lemmas}
\crefname{remark}{Remark}{Remarks}
\crefname{prop}{Proposition}{Propositions}
\crefname{defn}{Definition}{Definitions}
\crefname{corollary}{Corollary}{Corollaries}
\crefname{cor}{Corollary}{Corollaries}
\crefname{section}{Section}{Sections}
\crefname{figure}{Figure}{Figures}
\crefname{quest}{Question}{Questions}

\newcommand{\N}{\mathbb{N}}
\newcommand{\Z}{\mathbb{Z}}

\newcommand{\R}{\mathbb{R}}
\newcommand{\C}{\mathbf{C}}

\newcommand{\Poly}{\mathit{Poly}}
\newcommand{\oA}{\overline{A}}

\newtheorem{thm}{Theorem}[section]
\newtheorem{lemma}[thm]{Lemma}
\newtheorem{prop}[thm]{Proposition}

\newtheorem{quest}[thm]{Question}

\theoremstyle{definition}
\newtheorem{definition}[thm]{Definition}

\newtheorem{remark}{Remark}

\begin{document}
	\begin{abstract}  
    Weyl's classical equidistribution theorem states that  real valued polynomial sequences are uniformly distributed modulo 1,  unless all non-constant coefficients  are rational. A  continuous function  between two topological groups  is called a \emph{polynomial map} of degree at most $d$ if it vanishes under any $d+1$ difference operators.
    Leibman, and subsequently Green and Tao,  formulated and proved  equidistribution theorems about polynomial sequences that take values in a nilmanifold. 
    We  formulate and prove some general equidistribution  theorems regarding polyonomial maps from a locally compact group into a compact abelian group. 
    \end{abstract}

\maketitle

\section{Introduction}

Weyl's classical equidistribution theorem says that for any polynomial $P$ with real coefficients, the fractional part of the sequence  $(P(n))_{n=1}^\infty$ is uniformly distributed modulo 1, unless all non-constant coefficients of $P$ are rational. In the later case, the sequence is periodic modulo $1$ and  takes rational values, up to an additive constant.

The following natural notion of ``polynomial mappings of groups'' was studied by Leibman \cite{MR1910931}: A  continuous function  between two topological groups  is called a \emph{polynomial map} of degree at most $d$ if it vanishes under any $d+1$ difference operators.  See the following section for details. It is not difficult to check that any  polynomial map  $P:\R \to \R$ of degree at most $d$ in Leibman's sense above is of the form $P(x)=\sum_{i=0}^d a_i x^i$ for some $a_0,\ldots,a_d \in \R$. 
Thus, Weyl's equidistribution theorem can be interpreted as a uniform distribution result about polynomial maps $P:\Gamma_1 \to \Gamma_2$ where $\Gamma_1=\Z$ and $\Gamma_2 = \R/\Z$. Weyl's original paper \cite{MR1511862} already contains equidistribution results for ``multidimensional'' and ``multivariate'' polynomials, or equivalently equidistribution results for polynomial maps $P:\Gamma_1 \to \Gamma_2$ where $\Gamma_1= \Z^{d_1}$ and $\Gamma_2= (\R/\Z)^{d_2}$.
Leibman formulated and proved  equidistributions theorems for polynomial sequences taking values in a compact nilmanifold  \cite{MR2122919}. Subsequently, Leibman extended his equidistributions theorems for ``multivariate sequences'' 
 taking values in a compact nilmanifold \cite{MR2122920}.

The purpose of this note is to formulate and prove  the following generalizations of Weyl's  theorem:

\begin{thm}
\label{thm:poly_to_torus_well_distributed_or_rational}
Let $\Gamma$ be a locally compact metrizable amenable group, and let $P:\Gamma \to \R/\Z$ be a polynomial map. 
Then either $P$ is totally well-distributed with respect to Haar measure on $\R/\Z$, or $P$ takes values in a coset of a finite subgroup of $\R/\Z$. 
\end{thm}

The following theorem is a generalization  of the ``multidimensional Weyl theorem'' regarding polynomials maps from a finitely generated group $\Gamma$ into $\R/\Z$:
\begin{thm}\label{thm:Weyl_finitely_generated_to_d_torus}
Let $\Gamma$ be a finitely generated amenable group, and let $P:\Gamma \to (\R/\Z)^d$ be a polynomial map.

Then there exists a  subgroup $\Gamma_1 < \Gamma$ such that $[\Gamma:\Gamma_1] < + \infty$ and a  
closed, connected subgroup $T \subseteq (\R/\Z)^d$ and a function $\alpha:\Gamma/ \Gamma_1 \to (\R/\Z)^d/T$ 
such that:
\begin{enumerate}
\item For any  $\gamma \in \Gamma$ we have $\overline{P(\gamma \Gamma_1)}=T + \alpha(\gamma \Gamma_1)$.
\item For any  $\gamma \in \Gamma$, let $P_\gamma:\Gamma_1 \to T$ be given by $P_\gamma(\tilde \gamma)= P(\gamma \tilde \gamma)-\alpha(\gamma \Gamma_1)$. Then $P_\gamma$ is well-distributed with respect to Haar measure on $T$.
\end{enumerate}
\end{thm}

\Cref{thm:Weyl_finitely_generated_to_d_torus} says that ``up to finite index'' a polynomial map from a finitely generated group into a (finite-dimensional) torus is well-distributed with respect to Haar measure.

Considering polynomial maps from a finitely generated group $\Gamma$ into a compact abelian group $G$ we have the following statement:

\begin{thm}\label{thm:finitely_generated_to_compact_poly_uniquely_ergodic}
Let $\Gamma$ be a finitely generated group, let $G$ be a compact abelian group, and let $P:\Gamma \to G$ be a polynomial map.  Then $P$ is uniquely ergodic.
\end{thm}
Compared \Cref{thm:Weyl_finitely_generated_to_d_torus}, the statement of \Cref{thm:finitely_generated_to_compact_poly_uniquely_ergodic} does not involve amenability of $\Gamma$, and provides an assertion about the distribution of the orbit of  $P$  in the space of functions  $G^\Gamma$ rather than just the projection of this distribution onto $G$. This comes at a price: The statement of \Cref{thm:finitely_generated_to_compact_poly_uniquely_ergodic} does not provide an explicit description of the  limiting distribution. 

We will say a bit more about this in the final section.

 Different extensions and refinements of Weyl's equidistribution theorem can be found in the literature. 
For instance, Boshernitzan \cite{MR1269206} obtained a uniform distribution modulo 1 
for a class of sequences that includes  $(P(\sqrt[k]{n}))_{n=1}^\infty$, where $P$ is a polynomial (excluding certain ``rational'' degenerate cases).
Bergelson and Leibman obtained an analog of Weyl's equidistribution theorem in finite characteristic \cite{MR3439703}, which is a uniform distribution theorem for (a special class of) polynomial maps $P:\Gamma_1 \to \Gamma_2$, where $\Gamma_1$ is the discrete countable abelian group of polynomials over a finite field, and $\Gamma_2$ is the compact group obtained as the quotient of  formal Laurent series by polynomials over a finite field.
Bergelson and Moreira introduced an ``adelic'' version of Weyl's equidistribution theorem \cite[Theorem 5.2]{MR3479166}.

At first glance, it might seem surprising that  ergodic theoretic results about polynomial maps can be obtained under such general assumptions on the group $\Gamma$. This seeming generality is somewhat deceiving: Liebman has shown that any polynomial map which takes values in a nilpotent is obtained by lifting a  function defined on a nilpotent quotient \cite{MR1910931}. Liebman's result allows us to easily derive \Cref{thm:finitely_generated_to_compact_poly_uniquely_ergodic} from \Cref{thm:poly_to_torus_well_distributed_or_rational}. 

The paper is organized as follows:

In \Cref{sec:prelim}, we introduce the basic notions involved in the statement of our results. 
 We also state and prove a basic lemma relating the notions of well-distirbution and unique ergodicity.
In \Cref{sec:homo_non_abelian} we prove well-distirbution and unique ergodicity for polynomial maps of degree $1$. In the case of degree $1$ polynomials, the statement of the results and the proofs are relatively straightforward and probably well-known. We bring them mostly for context and  to contrast the case of degree $1$ polynomials with the case of higher degree.
\Cref{sec:poly_comb} discusses relatively simple combinatorial properties of polynomial maps that will be applied in our later proofs. The discussion in \Cref{sec:poly_comb} is elementary, in the sense that it does not involve topology, analysis, or measure theory (in fact, it requires nothing beyond the basic definition of a group).
In \Cref{sec:poly_well_dist} we prove \Cref{thm:poly_to_torus_well_distributed_or_rational} and \Cref{thm:Weyl_finitely_generated_to_d_torus}. In \Cref{sec:uniqu_ergo_poly}
we deduce \Cref{thm:finitely_generated_to_compact_poly_uniquely_ergodic}, using the previous results. \Cref{sec:concluding} we discuss possible extensions, refinements, generalizations, and additional remarks.

I thank Yaar Solomon and Matan Tal for their remarks on early versions and Vitaly Bergelson for valuable comments, references, and encouragement to make this writing publicly available.

\section{Preliminaries}\label{sec:prelim}
We introduce some notation and clarify definitions of various terms in the statement of the results.

We use multiplicative notation for the group operation in a locally compact metrizable topological  group $\Gamma$, whenever we do not assume $\Gamma$ is commutative. We let $m_\Gamma$ denote the left Haar measure on $\Gamma$. For a compact group, we always assume that the Haar measure is normalized to be a probability measure. In the case that $\Gamma$ is discrete, we normalize sothat $m_\Gamma(1_\Gamma)=1$, so $m_\Gamma$ is simply the counting measure.  Otherwise, when $\Gamma$ is not compact, $m_\Gamma(\Gamma)= + \infty$ and  we assume some (arbitrary) choice of normalization for $m_\Gamma$. 

\begin{definition}
    Given a measurable set $F \subseteq \Gamma$ with $0 < m_\Gamma(F) < +\infty$
    and a measurable subset $A \subseteq \Gamma$ we denote:
    \[
    m_F(A)  = \frac{m_{\Gamma}(A \cap F)}{m_\Gamma(F)}.
    \]
\end{definition}
Then $m_F$ is a Borel probability measure on $\Gamma$.

For a compact abelian group $G$, we usually
 use additive notation for the group operation. The one  exception to this convention is the case where $G = \mathbb{S}^1 = \{ z \in \mathbb{C}^* :~ |z|=1\}$ is the unit sphere in $\mathbb{C}$, in which case the group operation is the multiplication of complex numbers, for which we use multiplicative notation. 
Also, for a compact abelian group $G$ we let $\widehat{G}$ denote the  group of characters of $G$, or the Pontryagin dual group, namely the continuous homomorphism from $G$ to $\mathbb{S}^1$ .

If $G$ is a compact abelian group, $H <G$ is a closed subgroup  and $C = v+H$ is a coset of $H$ (where $v \in G$ is some element of the group), by ``the Haar measure on $C$'' we mean the unique probability measure on $C$ which is invariant under translation by elements of $H$. Equivalently, the Haar measure on the coset $C$ is the pushforward of the Haar measure on the closed subgroup $H$ via the map $x \mapsto x+v$ from $H$ to $C$.

In the following, $\Gamma$ will denote a locally compact, metrizable  topological  group. 
\begin{definition}
For subset $F_0, F \subset \Gamma$ and $\gamma \in \Gamma$ we denote:
\[
\partial_{\gamma} F = \gamma F \triangle F \mbox{ and }
\partial_{F_0} F = \bigcup_{\gamma \in F_0} \partial_{\gamma} F.
\]

For $\gamma \in \Gamma$ let $\sigma_\gamma:G^\Gamma \to G^\Gamma$ be given by $\sigma_\gamma(\Phi)(x):=\Phi(\gamma^{-1}x)$.

A sequence $(F_n)_{n=1}^\infty$ of measurable subset of $\Gamma$ having finite non-zero Haar measure is called a \emph{F\o lner sequence} if for every $\gamma \in \Gamma$ it holds that:
\[
\lim_{n \to \infty} \frac{m_\Gamma\left(\partial_\gamma F_n\right)}{m_\Gamma( F_n)} = 0.
\]
\end{definition}

equation 

\begin{definition}
Let $\Gamma$ and $G$ be groups, with the group operation written in multiplicative notation.
Given  a function $f:\Gamma \to G$ and $\gamma \in \Gamma$, let  $\Delta_\gamma f: \Gamma \to G$ denote the discrete derivative of $f$ with respect to $\gamma$:
\begin{equation}\label{eq:Delta_f_def}
    \Delta_\gamma f(x) = f(\gamma x)(f(x))^{-1} \mbox{ for } x \in \Gamma.
\end{equation}

If $G$ is abelian and additive notation is used for the group operation, \Cref{eq:Delta_f_def} takes the following form:
\[
\Delta_\gamma f(x) = f(\gamma x) -f(x) \mbox{ for } x \in \Gamma.
\]
\end{definition}
Note that for functions $\Phi:\Gamma \to \mathbb{S}^1$, since the group operation is written as multiplication,  we have 
\[
\Delta_\gamma \Phi(x) = \Phi(\gamma x)\overline{\Phi(x)},
\]
where $\overline{\Phi(x)}$ is the complex conjugate of $\Phi(x)$.

\begin{definition}
A continuous function
$P:\Gamma \to G$ is called a \emph{Polynomial map of degree at most $d-1$} if 
\[ \Delta_{\gamma_d}\ldots  \Delta_{\gamma_1} P =0 ~\forall \gamma_1,\ldots,\gamma_d \in \Gamma.
\]
We denote by $\Poly_d(\Gamma,G)$ the set of polynomial maps from of degree at most $d$ from $\Gamma$ to $G$.
\end{definition}
In particular, a polynomial map of degree $0$ is a constant map from $\Gamma$ to $G$.

A sequence $(a_n)_{n=1}^\infty$ is said to be \emph{equidistributed}  with respect to a probability measure $\mu$ on $\mathbb{R}$ if the sequence of measures $\frac{1}{N}\sum_{n=1}^N \delta_{a_n}$ converge weak-$*$ to $\mu$.
A sequence $(a_n)_{n=1}^\infty$ is said to be \emph{well-distributed} with respect to a probability measure $\mu$ on $\mathbb{R}$ if for every sequence of intervals $(I_j)_{j=1}^\infty$ whose length tends to infinity, the sequence of measures  $\frac{1}{|I_j|}\sum_{n \in I_j}\delta_{a_n}$ converge  to $\mu$ with respect to the weak-* topology. It is obvious that any well-distributed sequence is also equidistributed.
The notion of well-distribution can be naturally generalized to functions from an amenable group $G$ to a compact metrizable topological space $X$ as follows:

\begin{definition}
Let $\Gamma$ be a locally compact amenable group.

For $\phi \in L^\infty(\Gamma)$ and a measurable set $F \subset \Gamma$ with $0 < m_\Gamma(F) < \infty$, we denote
\[
A_F(\phi) :=  \frac{1}{m_\Gamma(F)}\int_{F} \phi(x)dm_\Gamma(x) = \int \phi(x) dm_F(x).
\]
\end{definition}

\begin{definition}
 We say that  $f \in L^\infty(\Gamma)$ has \emph{mean $z_0 \in \mathbb{C}$} if 
for any F\o lner sequence $(F_n)_{n=1}^\infty$ in $\Gamma$ the following holds:
\[
\lim_{n \to \infty}A_{F_n}(f)  = z_0.
\]
Equivalently, $f \in L^\infty(\Gamma)$ has mean $z_0$ if $\Lambda(f)=0$ for any $\Gamma$-invariant mean  $\Lambda$ on $L^\infty(\Gamma)$.
\end{definition}

\begin{definition}
For $f \in L^\infty(\Gamma)$, we denote
\[\oA_\Gamma(f) = \sup \left\{ \limsup_{n \to \infty} A_{F_n}(f) :~ (F_n)_{n=1}^\infty \mbox{ is a F\o lner sequence} \right\} \]
\end{definition}

\begin{definition}
We say that a Borel subset $B \subseteq \Gamma$ \emph{has density} $p \in [0,1]$ if 
for any F\o lner sequence $(F_n)_{n=1}^\infty$ in $\Gamma$ the following holds:
\[
\lim_{n \to \infty}m_{F_n}(B) = p.
\]
Equivalently, $B \subseteq \Gamma$ has density $p$ if and only if the indicator function of $B$ has mean $p$.
\end{definition}

\begin{definition} 
Let $\phi:\Gamma \to X$ be a Borel measurable function from a locally compact amenable group $\Gamma$ to a compact metrizable space $X$, and let $\mu$ be a Borel probability measure on $X$.
We say that $\phi$ is \emph{well-distributed with respect to $\mu$} if for every continuous function $f:X \to \mathbb{C}$ the mean of $f \circ \phi:\Gamma \to \mathbb{C}$ is $\int f d\mu$. 
\end{definition}

In other words, a function $\phi:\Gamma \to X$ is well distributed with respect to a probability measure $\mu$ on $X$ if  for any F\o lner sequence $(F_n)_{n=1}^\infty$ in $\Gamma$ the measures $\int_{F_n}\delta_{\phi(\gamma)} dm_{F_n}(\gamma)$ converge weak-$*$ to $\mu$.

\begin{definition}
A Borel measurable function $\phi:\Gamma \to X$ is \emph{totally well-distributed with respect to $\mu$} if for every finite-index subgroup $\Gamma_0 < \Gamma$ and every $\gamma \in \Gamma$, the restriction of $\gamma \phi$ to $\Gamma_0$ is well-distributed (as a function from $\Gamma_0$ to $X$).
\end{definition}

Given a locally compact Polish group $\Gamma$ and a compact topological space $X$, we denote by $C(\Gamma,X)$ the space of continuous functions from $\Gamma$ to $X$. The space $C(\Gamma,X)$, equipped with the topology of uniform convergence on compact sets and  is a itself a locally compact Polish topological space. If  $G$ is a compact metrizable group, $C(\Gamma,G)$ is a locally compact Polish group with respect to the operation of pointwise multiplication (in $G$).
The group $\Gamma$ acts on $C(\Gamma,X)$ by $(\gamma \cdot \phi)(\tilde \gamma)= \phi(\gamma^{-1} \tilde \gamma)$, $\phi \in C(\Gamma,X)$, $\gamma,\tilde \gamma \in \Gamma$. 
A probability measure $\mu$ on $C(\Gamma,X)$ is called \emph{$\Gamma$-invariant} when $\gamma_* \mu= \mu$ for every $\gamma \in \Gamma$.

In the case $\Gamma$ is countable and discrete, it holds that $C(\Gamma,X)= X^\Gamma$ is the space of functions from $\Gamma$ to $X$. In this case, the topology of uniform convergence on compact sets is just the product topology on $X^\Gamma$, so $X^\Gamma$ is compact  by Tychonoff's theorem. 
Given $\phi \in C(\Gamma,X)$, the $\Gamma$-orbit of $\phi$ is 
\[
\Gamma \phi = \left\{ g \cdot \phi:~ g \in \Gamma\right\}.
\]
The orbit-closure of $\phi$ (with respect to the action of $\Gamma$) is $\overline{\Gamma \phi}$, namely, the smallest closed set in $C(\Gamma,X)$ that contains the $\Gamma$-orbit of $\phi$.

\begin{definition}
Let $\Gamma$ be a locally compact group and $X$ a compact metizable topological space.
A continuous function $\phi \in C(\Gamma,X)$ is called \emph{uniquely ergodic} if there exists a unique $\Gamma$-invariant probability measure on $C(\Gamma,X)$ whose support is contained in  the orbit closure $\overline{\Gamma \phi}$.
\end{definition}

Given $\Phi \in C(\Gamma,X)$ and a compact subset $F \subset \Gamma$ let  $\Phi^F:\Gamma \to C(F,X)$ be defined by $\Phi^F(\gamma)= ( \gamma \Phi)\mid_{F}$ for all $\gamma \in \Gamma$.

We formulate the following simple lemma that relates unique ergodicity  and well-distribution:

\begin{lemma}\label{lem:unique_ergodicity_implies_equidistribution}
Let $\Gamma$ be a locally compact amenable group,  let $X$ be a topological space,  
and let $\Phi \in C(\Gamma,X)$ be a continuous function from $\Gamma$ to $X$. 
Suppose that the orbit closure of $\Phi$ is compact.
Then $\Phi$ is uniquely ergodic if and only if for every finite set $F \subseteq \Gamma$, the function $\Phi^F:\Gamma \to C(F,X)$ is well-distributed.
In that case, if $\mu$ is the unique $\Gamma$-invariant probability measure on $\overline{\Gamma \phi}$, then for every finite set $F \subseteq \Gamma$, the function $\Phi^F:\Gamma \to C(F,X)$ is well-distributed with respect to the pushforward of $\mu$ via the natural projection $C(\Gamma,X) \mapsto C(F,X)$.
\end{lemma}
\begin{proof}
Suppose that $\mu$ is the unique $\Gamma$-invariant probability measure on $\overline{\Gamma\Phi}$. 
Give a finite set $F \subseteq \Gamma$, let $\mu_F$ denote pushforward of $\mu$ via the restriction map from $X^\Gamma$ to $X^F$.
 We will prove that $\Phi^F$ is well-distributed by showing that for every F\o lner sequence $(F_n)_{n=1}^\infty$ in $\Gamma$,  the sequence of measures $\int_{F_n} \delta_{g \Phi^F}dm_{F_n}(g)$ converges weak-$*$  to $\mu_F$. 
 The assumption that $\overline{\Gamma \Phi}$ is compact implies that $\overline{ \Phi^F(\Gamma)}$ is a compact subset of $C(X,F)$. Thus, for every F\o lner sequence $(F_n)_{n=1}^\infty$ in $\Gamma$ there is a subsequence such that $\int_{F_n} \delta_{g \Phi^F}dm_{F_n}(g)$ converges weak-$*$  to some probability measure on  $\overline{\Gamma \Phi}$. If $\mu_F$ is the only limit point, we are done.
 Suppose by contradiction that this is not the case. Passing to a subsequence we can assume that $\int_{F_n} \delta_{g \Phi^F}dm_{\Gamma}(g)$ converges weak-$*$  to a probability measure $\nu_1 \ne \mu_{F}$ on $C(F,X)$. Starting with $K_1=F$, fix a an increasing sequence  $K_1 \subseteq K_2 \subseteq \ldots \subseteq K_j \subseteq \ldots$  of compact subsets
 such that $\bigcup_{j=1}^\infty K_j = \Gamma$.  Passing to a subsequence of $(F_n)_{n=1}^\infty$ 
 at every step, and applying the Cantor diagonal argument, we can assume that for every $j \in \mathbb{N}$ $\frac{1}{m_\Gamma(F_n)}\int_{F_n} \delta_{\Phi^{K_j}(g)}$  converges to a probability measure  $\nu_j$ on $C(K_j,X)$, so that $\nu_j \ne \mu_{K_j}$. By Caratheodory's extension theorem, there exists a unique probability measure $\nu$ on $C(\Gamma,X)$ such that the pushforward of $\nu$ via the restriction map $g \mapsto g\mid_{K_j}$ is equal to $\nu_j$. It follows that $\int_{F_n} \delta_{g \Phi}dm_{F_n}(g)$ converges to $\nu$, and so $\nu$ is $\Gamma$-invariant, but $\nu \ne \mu$. This contradicts unique ergodicity of $\overline{\Gamma\Phi}$. 

 Conversely, suppose that $\Phi^F$ is well-distributed for every compact subset $F$ of $\Gamma$. 
 Let $(F_n)_{n=1}^\infty$ be a F\o lner sequence in $\Gamma$. It follows that there exists a probability measure  $\mu$ on $C(\Gamma,X)$ such that for every  $\gamma \in \Gamma$ we have that $\int_{F_n}\delta_{\gamma \Phi} dm_{F_n}(\gamma) \to \mu$ as $n \to \infty$ . A measure $\mu$ as above is clearly $\Gamma$-invariant and supported on $\overline{\Gamma \Phi}$. From this it follows that for every F\o lner sequence and any $\tilde \Phi \in \overline{\Gamma \Phi}$ we have $ \int_{F_n}\delta_{\gamma \tilde \Phi} dm_{F_n}(\gamma) \to \mu$. Now by the ergodic theorem for any ergodic, $\Gamma$-invariant probability measure $\tilde \mu$ on $\overline{\Gamma \Phi}$ there exists a F\o lner sequence $(F_n)_{n=1}^\infty$  such with respect to $\tilde \mu$ almost every $\tilde \Phi$ satisfies $ \int_{F_n}\delta_{\gamma \tilde \Phi} dm_{F_n}(\gamma) \to \tilde \mu$.
\end{proof}

 We remark that the above proof of \Cref{lem:unique_ergodicity_implies_equidistribution} does not require a pointwise ergodic theorem: We can extract an almost-everywhere converging subsequence from the $L^2$ ergodic theorem. It follows that $\mu$ is the unique $\Gamma$-invariant probability measure on $\overline{\Gamma \Phi}$.
 
\section{Unique ergodicity for homomorphisms into  compact groups}\label{sec:homo_non_abelian}

In this section, we present a short and direct proof for unique ergodicity and well-distribution of any continuous homomorphism from a locally compact amenable group $\Gamma$ into a compact group $G$. This is a simple and general result. In particular, there is no need to assume that $G$ is abelian. The statement can be viewed as a generalized version of Kronecker's theorem. 

After a suitable reformulation, we can see that amenability of $\Gamma$ is irrelevant: We deduce the above well-distribution result from the following: The orbit-closure of any  continuous homomorphism from a locally compact group to a compact group $\Gamma$ is a coset of a compact group and that Haar measure is the unique invariant measure.
The proof is direct short and direct: We first verify that the orbit closure is a coset of a compact group, then show that any $\Gamma$-invariant measure is invariant under translations by elements of the corresponding compact subgroup, hence must be Haar measure. The well-distribution result follows easily from the unique ergodicity result for the orbit closure via \Cref{lem:unique_ergodicity_implies_equidistribution}. 
All the arguments are quite classical.
Interestingly, the structure of the proofs of our main results for higher degree polynomials (with compact abelian range) is in some sense opposite: The well-distribution result is proved directly using Weyl's well-distribution criterion via characters, from well-distribution, we deduce the algebraic structure of the closure of the range, and the unique ergodicity result follows by applying the well-distribution result to an induced polynomial map into $G^\Gamma$.

We first characterize orbit closures of homomorphisms in $C(\Gamma,G)$, and show that these are uniquely ergodic:

\begin{prop}\label{prop:homo_uniquely_ergodic}
Let $\Phi:\Gamma \to G$ be a continuous homomorphism from a locally compact metrizable  group $\Gamma$ to a compact metrizable group $G$. 
\begin{enumerate}
    \item The orbit closure of $\Phi$ is a right coset of 
    a compact subgroup of $C(\Gamma,G)$.  
    \item $\Phi$ is uniquely ergodic, and the unique $\Gamma$-invariant probability measure on $\overline{\Gamma \Phi}$ is the Haar measure on the coset $\overline{\Gamma \Phi}$.
\end{enumerate}

\end{prop}
\begin{proof}

For every $g \in G$ let $g^\Gamma \in C(\Gamma,G)$ denote the constant function given by $g^\Gamma(\gamma)  := g$ for every $\gamma \in \Gamma$. The map $g \mapsto g^\Gamma$ defines a continuous homomorphism from the group $G$ into $C(\Gamma,G)$.
Let $G_\Phi=\overline{\Phi(\Gamma)}$ denote the closure in $G$ of the image of $\Gamma$ under $\Phi$.

If $g = \lim_{n \to \infty}\Phi(\gamma_n)$ and $\tilde g = \lim_{n \to \infty}\Phi(\tilde \gamma_n)$ then $g \tilde g= \lim_{n \to \infty}\Phi( \gamma_n \tilde \gamma_n)$ and $g^{-1} = \lim_{n \to \infty}\Phi( \gamma_n^{-1})$.
This shows that $G_\Phi$ is a closed subgroup of $G$.
Let 
\[ \tilde G_\Phi := \left\{ g^\Gamma :~ g\in G_\Phi \right\}.\]
Then $\tilde G_\Phi$  is a compact subgroup of  $C(\Gamma,G)$, as it the image of the compact group $G_\Phi$ under a continuous homomorphism.

\begin{enumerate}
    \item For any $\gamma,x \in \Gamma$ we have $(\gamma^{-1} \cdot \Phi)(x) (\Phi(x))^{-1}= \Phi(\gamma)$, so $(\gamma^{-1} \cdot \Phi)\Phi^{-1} = (\Phi(\gamma))^\Gamma \in C(\Gamma,G)$. 

    It follows that
    $\Gamma \Phi = \left\{ \Phi(\gamma)^\Gamma \Phi :~ \gamma \in \Gamma\right\}$ and so \[\overline{\Gamma \Phi}= \left\{ g^\Gamma \Phi :~ g \in G_\Phi \right\} = \tilde G_{\Phi} \Phi,\]
    namely the the right coset of $\tilde G_\Phi$  that contains $\Phi$.
    \item For any $g \in G_\Phi$ and $\gamma,x \in \Gamma$  we have $\gamma \cdot  (g^\Gamma \Phi)(x)= (g^\Gamma \Phi) ( \gamma^{-1} x) = g \Phi(\gamma^{-1}) \Phi(x)$. We have shown that $\gamma \cdot (g^\Gamma \Phi) = g^\Gamma \Phi(\gamma^{-1})^\Gamma \Phi$. Let $m_{\Phi}$ denote the Haar measure on $\overline{\Gamma \Phi}= \tilde G_\Phi \Phi$. We first show that $m_\Phi$ is indeed a $\Gamma$-invariant measure. Let $f:\overline{\Gamma \Phi} \to \mathbb{R}$ be a bounded measurable function. Then
    \[ \int f( \phi) dm_{\Phi}(\phi) = \int f( x^\Gamma \Phi)dm_{G_\Phi}(x).\]
    Thus, for any $\gamma \in \Gamma$ 
    \[ \int f(\gamma \cdot \phi) dm_{\Phi}(\phi) = \int f( x^\Gamma \Phi(\gamma^{-1}) \Phi)dm_{G_\Phi}(x)=
     \int f( x^\Gamma \Phi)  dm_{G_\phi}(x).\]
    In the last equality we used the fact that the Haar measure $m_{G_\Phi}$ on the compact group $G_\Phi$ is also invariant with respect to multiplication from the right, because $G_\Phi$ is compact (hence unimodular).
    This show that $m_\Phi$ is indeed $\Gamma$-invariant.
   
    Conversely, let $\mu$ be a $\Gamma$-invariant probability measure on $\overline{\Gamma \Phi}= \tilde G_\Phi \Phi$. 
    Let $\pi:\tilde G_\Phi \Phi \to G_\Phi$ denote the homeomorphism given by $\pi(\gamma^\Gamma \Phi) = \gamma$. Then $\pi_* \mu$ is a probability measure on $G_\Phi$. Let $f:G_\Phi \to \mathbb{R}$ be a continuous function. 
    Then 
    \[ \int f( \pi^{-1}(\phi))d\mu(\phi) = \int f(x) d(\pi_* \mu)(x).\]
    As in the previous part, we have:
    \[ \int f( \gamma \cdot \pi^{-1}(\phi))d\mu(\phi) =  \int f( x \Phi(\gamma^{-1}) )d(\pi_*\mu)(x).\]
    By the assumption that $\mu$ is a $\Gamma$-invariant measure on $\tilde G_\Phi \Phi$ we conclude that 
    \[ \int f(x) d(\pi_* \mu)(x) = \int f( x \Phi(\gamma^{-1}) )d(\pi_*\mu)(x).\]
    Because $\Phi(\Gamma)$ is dense in $G_\Gamma$, it follows that for every $g \in G_\Gamma$,
    \[ \int f(x) d(\pi_* \mu)(x) = \int f( x g )d(\pi_*\mu)(x).\]
    This means that $\pi_* \mu$ is a proabability measure on the group $G_\Gamma$ that is invariant with respect to multiplication from the right by any element of $G_\Gamma$, so $\pi_* \mu$ is equal to Haar measure on $G_\Gamma$. By definition of the map $\pi$ this implies that $\mu$ is equal to Haar measure on $\tilde G_\Phi \Phi$.
\end{enumerate}

\end{proof}

From \Cref{prop:homo_uniquely_ergodic}, it is easy to deduce that homomorphisms from an amenable group into a compact group are well-distributed:

\begin{prop}\label{prop:homo_equidist}
Let $\Gamma$ be a locally compact amenable group, let $G$ be a compact group, and let $\Phi:\Gamma \to G$ be a continuous homomorphism.
Then:
\begin{enumerate}
    \item $\overline{\Phi(\Gamma)}$ is a closed subgroup of $G$.
    \item $\Phi$ is well distributed with respect to Haar measure on $\overline{\Phi(\Gamma)}$.
\end{enumerate}
\end{prop}
\begin{proof}
\begin{enumerate}
    \item By \Cref{prop:homo_uniquely_ergodic} $\overline{\Gamma \Phi}$ is a coset of a compact subgroup of $C(\Gamma,G)$. Let $\pi:C(\Gamma,G) \to G$ be given by $\pi(\phi)=\phi(1_\Gamma)$. Then $\pi$ is a continuous and surjective homomorphism and
    $\pi(\overline{\Gamma \Phi)}= \overline{\Phi(\Gamma)}$. It follows that $\overline{\Phi(\Gamma)}$ is an image of a coset of a subgroup of a compact subgroup under a homomorphism, thus it is a coset of a closed subgroup. But $1_G = \Phi(1_\Gamma)$, so $1_G \in \overline{\Phi(\Gamma)}$. It follows that $\overline{\Phi(\Gamma)}$ is the coset containing $1_G$, hence it is a compact subgroup.
    \item By  \Cref{prop:homo_uniquely_ergodic} Haar  measure is the unique $\Gamma$-invariant probability measure on $\overline{\Gamma \Phi}$. The pushforward of this measure by $\pi$ is equal to Haar measure on $\overline{\Phi(\Gamma)}$. By \Cref{lem:unique_ergodicity_implies_equidistribution} this implies that $\Phi$ is equidistributed with respect to Haar measure on $\overline{\Phi(\Gamma)}$.
\end{enumerate}
\end{proof}

\section{Combinatorial properties of  Polynomial}\label{sec:poly_comb}
In the following section, we present several simple and purely algebraic or combinatorial results regarding polynomial maps. We will use these facts in the proof of our main result. 

Given a subset $S$ of a group $\Gamma$ and $n \in \N$ we write 
\[ S^{\le n} = \left\{ s_1\cdot \ldots \cdot s_\ell :~ 0 \le \ell \le n, s_1,\ldots,s_\ell \in S\right\}.\]
For convince  we denote $S^{\le 0} = \{1_\Gamma\}$.

We recall the following result of Leibman: 
\begin{prop}[See, \cite{MR1910931}  Proposition 1.15]\label{prop:fg_poly_deteremined_by_finite_set}
Let $S$ be a generating set for the group $\Gamma$. 
Any $P \in \Poly_d(\Gamma,G)$ is uniquely determined by its values on $S^{\le d}$. 
\end{prop}
We include a short proof for completeness:
\begin{proof}
    The proof is based on some elementary general observations that hold for an arbitrary function $f:\Gamma \to G$:
    \begin{itemize}
       
        \item \emph{Observation 1:} The restriction of $f\mid_{S^{\le (n+1)}}$ determines $\Delta_s f\mid_{S^{\le n}}$ for all $s \in S$. This is obvious from the equation $\Delta_s f(g) = f(sg)(f(g))^{-1}$.
         \item \emph{Observation 2:} For every subset $W \subset \Gamma$,  $f\mid_W$ together with $\Delta_s f\mid_W$ determine $f\mid_{sW}$. This is obvious from  $f(sw) = \Delta_s f(w)f(w)$.
         \item \emph{Observation 3:} For every subset $W \subset \Gamma$,  $f\mid_W$ together with $\Delta_s f\mid_{s^{-1}W}$ determine $f\mid_{s^{-1}W}$. This is obvious from the equation $f(s^{-1}w) = (\Delta_s f(s^{-1}w))^{-1}f(w) $.
         \item \emph{Observation 4:} If $S \subseteq \Gamma$ is a generating set, then $\{\Delta_s f\}_{s \in S}$ together with $f(1_\Gamma)$ determine $f$. Indeed, by observations 2 and 3 the subset of $\Gamma$ that is determined by  $\{\Delta_s f\}_{s \in S}$ together with $f(1_\Gamma)$ is invariant under multiplication from the left by elements of $S \cup S^{-1}$, hence equal to $\Gamma$.   
         \end{itemize}

    We prove the lemma by induction on the degree $d$. The cases $d=1,0$ are obvious because a constant map is uniquely determined by its value on a single element, and a group homomorphism is uniquely determined by its values on a generating set.
    Let $P:\Gamma \to G$ be a polynomial map of degree at most $d+1$. By the induction hypothesis, for every $s \in S$, since $\Delta_s P$ is a polynomial map of degree at most $d$, it is uniquely determined by $\Delta_s P \mid_{S^{\le (d)}}$. 
    By observation 1, $P \mid_{S^{\le (d+1)}}$ determines $\Delta_s P \mid_{S^{\le (d)}}$. Hence $P \mid_{S^{\le (d+1)}}$ determines $\Delta_s P$ for all $s \in S$. By observation $4$, $P \mid_{S^{\le (d+1)}}$ determines $P$.
\end{proof}

\begin{prop}\label{prop:fg_to_finite_periodic}
Let $\Gamma$ be a finitely generated group and $G$ be a finite group.
Any $P \in \Poly_d(\Gamma,G)$ is periodic, in the sense that the orbit of $P$ under $\Gamma$ is finite. Moreover, if $S$ is a finite generating set for $\Gamma$, then $|\Gamma P| \le |G|^{|S^{\le d}|}$.
\end{prop}

\begin{proof}
Let $S$ be a finite generating set for $\Gamma$, and  $P \in \Poly_d(\Gamma,G)$. 
By \Cref{prop:fg_poly_deteremined_by_finite_set}, for every $\gamma \in \Gamma$ the function $\gamma \cdot P:\Gamma \to G$ is uniquely determined by $(\gamma P)\mid_{S^{\le d}}$. Hence the size of the orbit of $P$ under $\Gamma$ is bounded by $|G|^{|S^{\le d}|}$.
\end{proof}

\begin{prop}\label{prop:poly_finite_exponent}
    Let $\Gamma$ be an arbitrary group and $G$ an abelian group. For any $P \in \Poly_d(\Gamma,G)$ the following holds:
    \begin{enumerate}
        \item If $\gamma \in \Gamma$ and $n \in \N$ satisfy
        \begin{equation}\label{eq:Delta_gamma_n_zero}
             \Delta_{\gamma^n} P(x)  =0 \mbox{ for all } x \in \Gamma,
        \end{equation}
        then
        \[ n^d \Delta_\gamma P(x) = 0 \mbox{ for all } x \in \Gamma.\]
        \item If $\gamma \in \Gamma$, $n \in \N$ and $\alpha \in G$ satisfy 
        \begin{equation}\label{eq:n_Delta_gamma}
        n \Delta_\gamma P(x) = \alpha \mbox{ for all } x \in \Gamma,
        \end{equation}
        then 
        \[ \Delta_{\gamma^{n^d}} P(x) = n^d \alpha \mbox{ for all } x \in \Gamma.\]
    \end{enumerate}
\end{prop}
\begin{proof}
We will use the following formula:
\begin{equation}\label{eq:Delta_power_of_gamma}
    \Delta_{\gamma^k}P(x)= \sum_{j=0}^{k-1}\Delta_\gamma P(\gamma^j x).
\end{equation}
We will prove each of the claims by induction on the degree $d$.  When $d=0$, any polynomial map is constant; hence both claims are trivial. Suppose we proved  both claims for any $Q \in \Poly_{d-1}(\Gamma,G)$. 
\begin{enumerate}
    \item   Suppose $P \in \Poly_d(\Gamma,G)$  satisfies \eqref{eq:Delta_gamma_n_zero}.
    Let $Q= \Delta_\gamma P$. Then by the induction hypothesis, $n^{d-1} \Delta_\gamma Q(x)=0$ for every $x \in \Gamma$. Equivalently, $n^{d-1}Q(\gamma^j x) = n^{d-1}Q(x)$ for all $x \in \Gamma$.
    By \eqref{eq:Delta_gamma_n_zero} and \eqref{eq:Delta_power_of_gamma}
    \[ 0 = \Delta_{\gamma^n} P(x) = \sum_{j=0}^{n-1} \Delta_\gamma P(\gamma^j x).\]
     (using commutativity of $G$, we can multiplying the above equation by $n^{d-1}$  and we obtain:
     \[ 0 = \sum_{j=0}^{n-1} n^{d-1}Q(\gamma^j x) .\]
     Since $n^{d-1}Q(\gamma^j x) = n^{d-1}Q(x)$ for all $x \in \Gamma$ we conclude that $0 = \sum_{j=0}^{n-1} n^{d-1} Q(x)$,
     So $n^{d} \Delta_\gamma P(x) = n^{d-1} Q(x) =0$.
     \item  Suppose $P \in \Poly_d(\Gamma,G)$ satisfies \eqref{eq:n_Delta_gamma}. 
     Using \eqref{eq:Delta_power_of_gamma} we get
     \[ \Delta_{\gamma^{n^d}} P(x) = \sum_{i=0}^n \Delta_{\gamma^{n^{d-1}}}P(\gamma^{n^{d-1} i} x).\]
     By the induction hypothesis applyied on $Q = \Delta_{\gamma^{n^{d-1}}}P$, it follows that $\Delta_{\gamma^{n^{d-1}}}P(\gamma^{n^{d-1} i} x)= \Delta_{\gamma^{n^{d-1}}}P( x)$ for all $x \in \Gamma$.
     Since
     \[  n \Delta_{\gamma^{n^{d-1}}}P(x) = \sum_{j=0}^{n^{d-1}-1} n \Delta_\gamma P(\gamma^j x),\]
     We conclude that $\Delta_{\gamma^{n^d}} P(x) = n^d \alpha$. 
\end{enumerate}

\end{proof}
      In particular, \Cref{prop:poly_finite_exponent} shows that a polynomial map from a finite group $\Gamma$ to $(\R/\Z)^d$ takes values in a coset of a finite subgroup. We will use this fact later.
      This raises the following question:
      
\begin{quest}
     Let $P:\Gamma \to G$ be a polynomial map from a finite group $\Gamma$ into a group $G$. Is it always true that $P(\Gamma)$ is contained in a coset of a finite subgroup of $G$?
\end{quest}

\section{Polynomial maps are well-distributed}\label{sec:poly_well_dist}
In this section, we prove \Cref{thm:poly_to_torus_well_distributed_or_rational} and \Cref{thm:Weyl_finitely_generated_to_d_torus}.

The proof of \Cref{thm:poly_to_torus_well_distributed_or_rational} presented below is an adaptation Weyl's proof about well-distribution of polynomial sequences. In essence, it amounts to  replacing averages over ``long intervals'' in $\Z$ by averages over F\o lner sets in a locally compact group $\Gamma$. Paraphrasing  Green-Tao \cite[Section $4$, above Lemma 4.1 ]{MR2877065}, this is ``really just a reprise of the standard
theory of Weyl sums''. However, here there are slight subtleties due to ``exceptional cases'' such as the case where $\Gamma$ is finite, or compact. Unlike many other ``Weyl-type equdistribution theorems'', where  handling these ``exceptional cases'' is completely trivial, here some care is needed.

Let us introduce some terminology: 
\begin{definition}
    We say that a map $f:\Gamma \to \R/\Z$ is \emph{rational} if it takes values in a finite subgroup of $\R/\Z$. Equivalently, $f:\Gamma \to \R/\Z$ is rational if there exists $N \in \N$  such that $N P(\gamma) = 0$ for every $\gamma \in \Gamma$.
    We say that a map $f:\Gamma \to \R/\Z$ is \emph{off-rational} if it takes values in a coset of a finite subgroup of $\R/\Z$. Equivalently, $f:\Gamma \to \R/\Z$ is off-rational if and only if there exists $N \in \N$ and $\alpha \in \R/\Z$ such that $N P(\gamma) = \alpha$ for every $\gamma \in \Gamma$.
\end{definition}
Since any finite subgroup  of $\R/\Z$ is of the form $\{ \frac{k}{N} +\Z:~ k \in \N\}$ for some $N \in \N$, 
the following is a direct reformulation of \Cref{thm:poly_to_torus_well_distributed_or_rational}:
If $P:\Gamma \to \R/\Z$ be a polynomial map from a  locally compact amenable group  $\Gamma$ into $\R/ \Z$  that is not  totally well-distributed, then
$P$ is off-rational.

The following lemma is fairly obvious; we present a proof for completeness.
\begin{lemma}\label{lem:well-distirbuted_not_off_rational}
Let $\Gamma$ be a locally compact amenable group and $f:\Gamma \to \R/\Z$ a function. If there exists $\gamma \in \Gamma$ and a finite-index subgroup $\Gamma_0 <\Gamma$ such that $\gamma f \mid_{\Gamma_0}$ is off-rational, then $f$ is not well-distributed with respect to Haar measure on $\R/\Z$.
\end{lemma}
\begin{proof}
    Suppose that $\gamma f \mid_{\Gamma_0}$ is off-rational for some subgroup $\Gamma_0$ with $[\Gamma: \Gamma_0] < \infty$. 
    Thus there exists $\alpha \in \R/\Z$ and $n \in \N$  such that $n f(\gamma \tilde \gamma)=\alpha$  for every $\tilde \gamma \in \Gamma_0$.
    Then there for any F\o lner sequence $(F_n)_{n=1}^\infty$ there exists $k \in \{0,\ldots,n\}$ such that 
    \[
    \liminf_{j \to \infty}  m_{F_j}\left(\left\{\gamma \in F_j:~ f(\gamma) = \alpha +\frac{k}{n} \right\} \right) \ge \frac{1}{n [\Gamma:\Gamma_0]}.
    \]
     
    It follows that  for every Probability measure $\mu$ on $\R/\Z$  that is a limit of point of $\left( \int_{F_n} \delta_{f(\gamma)} dm_{F_n}(\gamma)\right)_{n=1}^\infty$ for some F\o lner sequence $(F_n)_{n=1}^\infty$ in $\Gamma$, there exists $\beta \in \R/\Z$ such that 
    $\mu(\{\beta\}) >0$ . In particular, $f$ is not well-distributed  with respect to Haar measure on $\R/\Z$.
\end{proof}

We now recall a classical criterion of Weyl for well-distribution. Weyl's criterion allows us to reduce the problem of well-distribution of a function taking values in a compact abelian group into convergence of averages of characters.
For a compact metrizable abelian group $G$ we denote by $\hat G$ the (countable discrete) group of continuous homomorphisms from $G$ to $\mathbb{S}^1 =\{z \in \mathbb{C}^*~:~ |z| =1\}.$ 

\begin{lemma}[Weyl's well-distribution criterion]\label{lem:Weyl_well_distribution_criterion}
A  continuous function $\Phi:\Gamma \to G$ from a locally compact amenable group to a compact metrizable abelian group $G$ is well-distributed with respect to a Borel probability measure $\mu$ on $G$  if and only if
$\chi \circ \Phi$ has mean $\chi(\mu) := \int \chi d\mu$  for every character $\chi \in \hat G$.
In particular, $\Phi:\Gamma \to G$ is well distributed with respect to Haar measure on $G$ if and only if $\chi \circ \Phi$ has mean zero with respect for every non-trvial character $\chi \in \hat G$.
\end{lemma}
\begin{proof}
Pontryagin duality tells us that  any continuous function $f:G \to \mathbb{C}$ can be uniformly approximated by linear combinations of  characters. Thus, $\Phi:\Gamma \to G$ is well distributed with respect to a Borel probability measure $\mu$ on $G$  if and only if $f \circ \Phi$ has mean $\int f d\mu$ whenever $f= \sum_{i=1}^k \alpha_k \chi_k$ is a linear combination of characters. This happens if and only if $\chi \circ \phi$ has mean $\int \chi d\mu$ for every character $\chi \in \hat G$. 

If $\chi$ is the trivial character $\chi \circ \phi$ is the constant function $1$, so it always has mean $1$. Since $\int \chi dm_G= 0$ for any non-trivial character $\chi \in \hat G$, the ``in particular'' statement follows. 
   
\end{proof}

The first lemma states that if a function into $\mathbb{S}^1$ has some constant but non-trivial directional derivative, then that function must have zero mean:
\begin{lemma}\label{lem:const_deriv}
Let  $\Phi:\Gamma \to \mathbb{C}$ be a bounded measurable function. If there exists $\gamma \in \Gamma$ and $z_0 \in \mathbb{S}^1 \setminus \{1\}$ such that $\Phi(\gamma x) = z_0$ for all $x \in \Gamma$ then $\Phi$ has mean zero.
\end{lemma}
\begin{proof}
Since $\Phi$ is bounded,  we can assume without loss of generality that $|\Phi(\gamma)| \le 1$ for all $\gamma \in \Gamma$.
Suppose that $z_0 \in \mathbb{S}^1 \setminus \{1\}$  and that $\Delta_\gamma \Phi(x) = z_0$ for all $x \in \Gamma$.
By definition, \[\Phi(\gamma x) = \Delta_\gamma \Phi(x) \Phi(x) = z_0 \Phi(x).\] So for every $k \in \Z$, we have that  $\Phi(\gamma^k x)= z_0^k \Phi(x)$.
Because $|\Phi(x)| \le 1$, for any compact set $F \Subset \Gamma$  we have
\[
\left| \int_F \Phi(x) dm_\Gamma(x) - \int_F \Phi( \gamma x) dm_\Gamma(x)  \right| \le  m_\Gamma(\partial_{\gamma}(F)).
\]

Substituting $\Phi(\gamma x) = z_0 \Phi(x)$ it follows that
\[
\left|
\int_F \Phi(x) dm_\Gamma(x) - \int_F z_0 \Phi(  x) dm_\Gamma(x)  \right| < m_\Gamma\left(\partial_{\gamma}(F)\right).
\]

Using $z_0 \ne 1$, we have:
\[
\left|
\int_F \Phi(x) dm_\Gamma(x) \right| < \frac{1}{|1-z_0|}m_\Gamma\left(\partial_{\gamma}(F)\right).
\]

Let $(F_n)_{n=1}^\infty$ be a F\o lner sequence in $\Gamma$. For each $n \in \N$ we have:
\[
\left|
\int_{F_n} \Phi(x) dm_{F_n}(x) \right| < \frac{1}{|1-z_0|}\frac{m_\Gamma\left(\partial_{\gamma}(F_n)\right)}{m_\Gamma(F_n)}.
\]

Taking $n \to \infty$ we conclude that $\Phi$ has mean zero.

\end{proof}

The main ingredient in the proof of \cref{thm:poly_to_torus_well_distributed_or_rational} is the following version of the van der Corput inequality. We reproduce the short proof for completeness  (see eg. \cite[Lemma 4.2]{MR1481813}):
\begin{lemma}[The Van der Corpout inequality for complex-valued functions on amenable groups]
\label{lem:VDC_trick}
Let $\Gamma$ be a locally compact group, and let $F,F_0 \subset \Gamma$ be Borel subsets of $\Gamma$ having finite positive Haar measure.
Let $\phi:\Gamma \to \mathbb{C}$ be a measurable function with $|\phi(\gamma)| \le 1$ for all $\gamma \in \Gamma$.
Then

\begin{equation}\label{eq:VDC_ineq}
| A_F(\phi) | \le \sqrt{ \int_{F_0} \int_{F_0} A_F\left(\sigma_{\gamma_2^{-1}}(\Delta_{\gamma_1^{-1} \gamma_2} \phi) \right)dm_{m_{F_0}}(\gamma_1)dm_{F_0}(\gamma_2) }+ \frac{m_\Gamma (\partial_{F_0} F)}{m_\Gamma(F)}
\end{equation}

\end{lemma}

\begin{proof}
Observe that for any $f:\Gamma \to \mathbb{C}$ with $|f| \le 1$ and any $F_0,F \Subset \Gamma$, $\gamma \in F_0$ we have
\begin{equation}\label{eq:integration_almost_invariant}
\left|A_F(f) - A_F(\sigma_\gamma f) \right| \le   \frac{m_\Gamma (\partial_{F_0}F)}{m_\Gamma(F)}.
\end{equation}

Taking $f=\phi$, averaging  over $\gamma \in F_0$ and using the triangle inequality we obtain
\[
\left| A_F(\phi) - A_F\left(\int_{F_0}\left(\sigma_{\gamma}\phi  \right) dm_{F_0}(\gamma)\right)\right| \le\frac{m_\Gamma (\partial_{F_0} F)}{m_\Gamma(F)}.
\]
 Applying the triangle inequality again and the Cauchy-Schwarz inequality we get:
\[
\left| A_F(\phi) \right| \le 
\left[A_F\left(\left|\int_{F_0}\sigma_{\gamma}\phi dm_{F_0}(\gamma) 
          \right|^2
      \right)\right]^{\frac{1}{2}} +\frac{m_\Gamma (\partial_{F_0} F)}{m_\Gamma(F)}.
\]
Expanding
\[
\left|\int_{F_0}\sigma_{\gamma}\phi dm_{F_0}(\gamma) 
          \right|^2 =  \left( \int_{F_0}\int_{F_0} \sigma_{\gamma_1}\phi \overline{\sigma_{\gamma_2}\phi} dm_{F_0}(\gamma_1) dm_{F_0}(\gamma_2) \right),
\]
Using that $\sigma_{\gamma_1} \phi \overline{\sigma_{\gamma_2}\phi} = \sigma_{\gamma_2^{-1}}(\Delta_{\gamma_1^{-1} \gamma_2} \phi)$ we get:

\[
| A_F(\phi) | \le \sqrt{ \int_{F_0} \int_{F_0} A_F\left(\sigma_{\gamma_2^{-1}}(\Delta_{\gamma_1^{-1} \gamma_2} \phi) \right)dm_{F_0}(\gamma_1)dm_{F_0}(\gamma_2) }+ \frac{m_\Gamma (\partial_{F_0} F)}{m_\Gamma(F)}
\]

\end{proof}

Using the inequality \eqref{eq:VDC_ineq}, once can deduce the following:

\begin{lemma}\label{lem:VDC_ap}

    Let $\Gamma$ be a locally compact amenable group, and let $\phi:\Gamma \to \mathbb{C}$ be a measurable function with $|\phi(\gamma)|\le 1$ for all $\gamma \in \Gamma$. For any $\epsilon >0$
    Let 
    \[
    K_\epsilon = \left\{ \gamma \in \Gamma:~ \oA_\Gamma(\Delta_\gamma \phi) \ge \epsilon^2/2 \right\}.
    \]
    
    If $\oA_\Gamma(\phi) \ge \epsilon$ then $\oA_\Gamma( 1_K) \ge \epsilon^2/2$.
\end{lemma}

The version of the well-known van der Corput Lemma below  follows directly (see \cite[Theorem 2.11]{MR3479166} for a closely related statement, with a somewhat different proof):

\begin{prop}[The van der Corput Lemma for maps from amenable groups to compact abelian groups]\label{lem:VDC}
Let $\Phi:\Gamma \to G$ be a function from  a locally compact amenable group $\Gamma$ to a compact abelian group $G$. Suppose there exists a subset $N \subset \Gamma$ with density zero  such that $\Delta_\gamma \Phi: \Gamma \to G$ is well distributed with respect to the Haar measure on $G$ for any $\gamma \in \Gamma \setminus N$. Then $\Phi$ is well-distributed with respect to the Haar measure on $G$.  
\end{prop}

\begin{remark}
For an infinite locally compact group  $\Gamma$, \Cref{lem:VDC} implies the following: If $\Delta_\gamma \Phi: \Gamma \to G$ is well-distributed with respect to the Haar measure on $G$ for any $\gamma \in  \Gamma \setminus \{1_\Gamma\}$, then $\Phi$ is well-distributed with respect to the Haar measure on $G$. The above implication may fail when $\Gamma$ is a finite group. For example, the map $f(n)=n^2 \mod 3$ viewed as a polynomial map $p:\Z/3\Z \to \Z/3\Z$ is not  uniformly distributed on $\Z/3\Z$, but $\Delta_m f$ is uniformly distributed, for any $m \in \Z/3\Z \setminus \{0\}$.
\end{remark}

We will need a slightly more quantitative statement: 

\begin{prop}\label{prop:finitary_poly_mean_zero}
    For every finitely generated amenable group $\Gamma$, $d \in \N$ and $\epsilon >0$ there exists $k = k(\Gamma,d,\epsilon) \in \N$ such that 
    for every $P \in \Poly_d(\Gamma,(\R/Z))$  and  $\chi \in \widehat{(\R/Z)}$, either
    $\oA_\Gamma(\chi \circ P) < \epsilon$ or there exists $\alpha \in \R/\Z$ such that $\chi(P(\gamma))^k=\alpha$ for all $\gamma \in  \Gamma$.    
\end{prop} 

\begin{proof}
We prove the lemma by induction on $d$. 
The case  $d=0$ is trivial to prove  with $k=1$. 
The case $d=1$ follows easily from \Cref{lem:const_deriv}.
Fix $\epsilon >0$.
Suppose $P \in \Poly_d(\Gamma,\R/\Z)$ satisfies $\oA_\Gamma(\chi \circ f) \ge \epsilon$.
Let
\[ K = \left\{ \gamma \in \Gamma:~ \oA_\Gamma( \chi \circ  \Delta_\gamma P )  \ge \epsilon^2/2 \right\}.\]
From \Cref{lem:VDC_ap} it follows that 
\[ \oA_\Gamma ( 1_K) \ge \epsilon^2/2.\]
By the induction hypothesis, there exists $k_1 = k(\Gamma,d-1,\epsilon^2/2)$ such that for every $\gamma \in K$ there exists $\alpha_\gamma \in \R/\Z$ such that $\chi(\Delta_\gamma P(x))^{k_1} = \alpha_\gamma$ for all $x \in \Gamma$.
By part $(2)$ of \Cref{prop:poly_finite_exponent}, it follows that $\chi(\Delta_{\gamma^{k_1^{d-1}}} P(x)) = (\alpha_\gamma)^{k_1^{d-1}}$ for every $\gamma \in K$. By \Cref{lem:const_deriv}, since $\oA_\Gamma(\chi \circ f) \ne 0$, it follows from \Cref{lem:const_deriv} that $(\alpha_\gamma)^{k_1^{d-1}} =1$ for every $\gamma \in K$. This means that for every $\gamma \in G$ and $x \in \Gamma$ $\chi (\Delta_{\gamma^{k_1^d}}P(x))=1$.
By part $(1)$ of \Cref{prop:poly_finite_exponent}, it follows that for every $\gamma \in K$ and $x \in \Gamma$ we have $\chi (\Delta_{\gamma}P(x))^{k_1^{d^2}}=1$.
Let
\[
\Gamma_0= \left\{ \gamma \in \Gamma:~  \chi(\Delta_\gamma P(x))^{k_1^{d^2}} =1 \right\}.\]
Then $K \subseteq \Gamma_0$ so $\oA(1_{\Gamma_0}) \ge \oA(1_G) \ge \epsilon^2/2$. Since $\Gamma_0$ is a subgroup,
$[\Gamma: \Gamma_0] = (\oA(1_{\Gamma_0}))^{-1} < 4\epsilon^{-2}$. Let $k_2 = [\Gamma: \Gamma_0]$. It follows that $\gamma^{k_2} \in \Gamma_0$ for every $\gamma \in \Gamma$. Thus, for every $\gamma,x \in \Gamma$, $\chi(\Delta_{\gamma^{k_2}} P(x))^{ k_1^{d^2}} =1$.
Using part $(1)$ of \Cref{prop:poly_finite_exponent}, for every $x,\gamma \in \Gamma$
we have $\chi(\Delta_{\gamma} P(x))^{ k_2^d k_1^{d^2}} =1$.
It follows that the claim we intended to prove holds with $k=k(\Gamma,d,\epsilon)$ equal to the least common multiple of  $k_2^d k_1^{d_2}$, where $k_2$ runs over all positive integers up to $4\epsilon^{-2}$.
\end{proof}

\begin{proof}[Proof of \Cref{thm:poly_to_torus_well_distributed_or_rational}]
Suppose $P \in P_d(\Gamma,\R/\Z)$ is not well-distributed with respect to Haar measure. By Weyl's well-distribution criterion, it follows that there exists a non-trivial character $\chi \in \widehat{\R/\Z}$ and $\epsilon >0$ such that $\oA_\Gamma(\chi \circ P) > \epsilon$. By \Cref{prop:finitary_poly_mean_zero} there exists $\alpha \in \R/\Z$ and $k \in \N$ such that $\chi(P(\gamma))^k = \alpha$ for every $\gamma \in \Gamma$. This proves that $P$ takes values in a coset of a finite subgroup of $\R/\Z$.

It remains to show that if $P$ is well-distributed,  it must be totally well-distributed. Let $\Gamma' < \Gamma$ be a finite-index subgroup. Then for every $\gamma \in \Gamma$, the restriction of $\gamma P$ to $\Gamma'$ is a polynomial map from $\Gamma'$ to $\R/\Z$. If there exists $\gamma \in \Gamma$ such that this restriction is off-rational, by \Cref{lem:well-distirbuted_not_off_rational} $P$ itself cannot be well-distributed  with respect to a Haar measure. 
\end{proof}

\begin{proof}[Proof of \Cref{thm:Weyl_finitely_generated_to_d_torus}]
Let $\Gamma$ be  a finitely generated group. Choose $P \in \Poly_d(\Gamma , (\R/\Z)^d)$. Let $\widehat{(\R/\Z)^d}$ denote the dual group. Then $\widehat{(\R/\Z)^d}$ is isomorphic to the group $\Z^d$. 
 
Let
\[
R_P = \left\{ \chi \in \widehat{(\R/\Z)^d}~:~ \exists z \in \mathbb{S}^1 \mbox{ and } n \in \N \mbox{ s.t. } \chi(P(\gamma))^n=z \mbox{ for every } \gamma \in \Gamma\right\}.
\]

Clearly, $R_P$ is a subgroup of $\widehat{(\R/\Z)^d} \cong \Z^d$, hence finitely generated.
From the fact that $R_P$ is finitely generated, it follows that there exists $n \in \Z$ such that $\chi^n(P(\Gamma))$ is a singleton for every $\chi \in R_P$.
Let
\[T = \{ x \in (\R/\Z)^d:~ \chi(x) = 1 \mbox{ for all }  \chi \in R_P\}.\] Then $T$ is a closed subgroup of $(\R/\Z)^d$, and $\widehat{T} \cong \widehat{(\R/\Z)^d)}/ R_P$. 
From the definition of $R_P$, we see that it is equal to its own radical in $\widehat{(\R/\Z)^d)}$, in the sense that whenever $\chi^k \in R_P$ for some $\chi \in (\R/\Z)^d$ and $k \in \N$ then $\chi \in R_P$. This implies that $\widehat{T} \cong \Z^k$ for some $0 \le k \le d$, so $T$ is a connected subgroup of $(\R/\Z)^d$.
We claim that $P(\Gamma)$ is contained in a finite union of  cosets of $T$. Indeed, if $R_P$ is generated by $\chi_1,\ldots,\chi_r$ and $\chi_i(z)=\chi_i(z')$ for all $1 \le i \le r$ then $z$ and $z'$ are in the same coset of $T$. Since $\chi_i(P(\gamma))^n=1$ for all $1\le i \le k$, it follows that $P(\Gamma)$ is contained in the union of at most $n^r$ cosets of $T$.
Let $\overline{P}:\Gamma \to (\R/\Z)^d/T$ be given by $\overline{P}(\gamma)=P(\gamma)+T$.
Then by \Cref{prop:fg_to_finite_periodic}, there exists a finite-index subgroup $\Gamma_1 < \Gamma$ such that $\overline{P}$ is constant on every coset of $\Gamma_1$.  Equivalently, every coset $\gamma \Gamma_1 \in \Gamma/\Gamma_1$ there exists a coset  $\alpha(\gamma \Gamma_1) \in (\R/\Z)^d/T$ such that $P(\gamma \Gamma_1) \subseteq T+ \alpha(\gamma \Gamma_1)$.
Thus, the function $P_\gamma:\Gamma_1 \to T$ is well-defined. It remains to check that $P_\gamma:\Gamma_1 \to T$ is well-distributed with respect to Haar measure on $T$ (this will show in particular that $P_\gamma(\Gamma_1)$ is dense in $T$).
For this we must show that $\chi \circ P_\gamma$ has mean zero for any non-trivial character $\chi \in \widehat{T}$. Since $\widehat{T} \cong \widehat{(\R/\Z)^d}/R_P$, we need to show that for every $\tilde \chi \in \widehat{(\R/\Z)^d} \setminus R_P$, the function $\tilde \chi \circ P_\gamma$ has mean zero. 
By \Cref{thm:poly_to_torus_well_distributed_or_rational} for every $\chi \in \widehat{(\R/\Z)^d}$ either $\chi \circ P$ takes values in a coset of a finite subgroup of $\R/\Z$, or $\chi \circ P$ is totally well-distributed. In the first case, by definition $\chi \in R_P$. In the second case, $\chi \circ P_\gamma$ has mean zero for every $\gamma \in \Gamma$.
\end{proof}

\section{Unique ergodicity of polynomial maps}\label{sec:uniqu_ergo_poly}

In this section, we prove \Cref{thm:finitely_generated_to_compact_poly_uniquely_ergodic}.
In what follows, $\Gamma$ will be a finitely generated group and $G$ a compact abelian group.
The space $G^\Gamma$ of functions from $\Gamma$ to $G$ is equipped with the product topology, which makes it into a compact topological space. It is furthermore a compact abelian group, with respect to the pointwise addition in $G$. The group $\Gamma$ acts on $G^\Gamma$ by homemorphisms (the shift action). This action is furthermore an \emph{algebraic action}: Each element of $\Gamma$ acts as an  automorphism of the compact  group $G^\Gamma$. A polynomial map $P: \Gamma \to G$ is by definition an element of $G^\Gamma$. We denote by $\overline{\Gamma P}$ the closure in $G^\Gamma$ of the orbit of $P$ under the action of $\Gamma$.

The proof of \Cref{thm:finitely_generated_to_compact_poly_uniquely_ergodic} can be completed as follows:

\begin{proof}[Proof of \Cref{thm:finitely_generated_to_compact_poly_uniquely_ergodic}]
Let $\Gamma$ be a finitely generated group, let $G$ is a compact abelian group and  $P \in \Poly_d(\Gamma,G)$.

By Leibman's result \cite[Proposition 3.21]{MR2877065} there exists a nilpotent group $\Gamma_0$, a surjective homomorphism $h:\Gamma \to \Gamma_0$ and a polynomial map $P_0:\Gamma_0 \to G$ such that $P = P_0 \circ h$. 
There is an obvious bijection between $\Gamma_0$-invariant measures on $\overline{\Gamma_0 P_0}$ and $\Gamma$-invariant measures on $\overline{\Gamma P}$. We can thus assume without loss of generality that $P=P_0$, $\Gamma=\Gamma_0$ and $h$ is the identity. In other words, we can assume that  $\Gamma$ itself is nilpotent.
Since nilpotent groups are amenable, by \Cref{lem:unique_ergodicity_implies_equidistribution} it suffices to prove that for any finite set $F \subset \Gamma$ the map $P^F:\Gamma \to G^F$  is well-distributed with respect to some probability measure on $G^F$. By Weyl's well-distribution criterion, it suffices to check  that $\chi \circ P^F:\Gamma \to \mathbb{S}^1$ is well-distributed for any finite subset $F \subset \Gamma$ and for any $\chi \in \widehat{G^F}$. But as  $\chi \circ P^F:\Gamma \to \mathbb{S}^1$ is a polynomial map, using that $\mathbb{S}^1 \cong \R/\Z$,
by \Cref{thm:poly_to_torus_well_distributed_or_rational} either $\chi \circ P^F$ is well-distributed or it takes values in a coset of a  finite subgroup of $\mathbb{S}^1$. If $\chi \circ P^F$ takes values in a coset of a finite subgroup of $\mathbb{S}^1$, we can use the fact that $\Gamma$ is finitely generated and apply \Cref{prop:fg_to_finite_periodic}, to conclude that the $\Gamma$-orbit of $\chi \circ P^F$ is finite. In particular, $\chi \circ P^F$ is well-distributed with respect to some (finite-valued) probability measure.
This completes the proof.
\end{proof}

\section{Concluding remarks and further directions}\label{sec:concluding}

We conclude with remarks regarding possible extensions, refinements, and  generalizations.

\subsection{Polynomial maps into nilmanifold}
Leibman's theorems from \cite{MR2122919} and \cite{MR2122920} are direct analogs  of the statement of \Cref{thm:Weyl_finitely_generated_to_d_torus}, but not a more general case, nor are they are restricted case of ours. On the one hand, Leibman's theorems have less restrictive assumptions regarding  the range of the polynomial map $P$. In Leibman's theorems, the range  is allowed to be a connected nilpotent Lie group $G$, rather than a torus. In the cases where $G$ is not compact, well-distribution is taken modulo a  co-compact subgroup $G_1 < G$. So that the range is $X:= G/G_1$ is a compact connected nilmanifold, generally a compact homogeneous space but not necessarily a compact group (compact connected nilpotent Lie groups are abelian). An alternative ``quantitative''  proof for Leibman's theorem can be found in \cite{MR2877065}. 
    On the other hand, in Leibman's theorems, the domain of the polynomial map is restricted to be $\Gamma= \Z$ (as in \cite{MR2122919}) or $\Gamma= \Z^d$ (as in \cite{MR2122920}), rather than an arbitrary finitely generated amenable group, as in the statement of \Cref{thm:Weyl_finitely_generated_to_d_torus}.
    We believe it should be possible to extend \Cref{thm:Weyl_finitely_generated_to_d_torus} to apply to the more general setting where the range of the map $P$  is a nilpotent Lie group $G$ (as in Leibman's theorem). The key for such generalization is the following: Rather than using Wey's well-distribution criterion, use the well-distibution criterion of Leon Green \cite{MR0167569} and Parry \cite{MR267558}. That is, when dealing with compact nilmanifolds rather then abelian  rather than compact abelian groups, one needs to check means of all \emph{central characters}, namely continuous functions $f:G/G_1 \to \C$ such that for every $g$ in the center of $G$ and every $x \in X$ satisfy $f(g x) = \lambda_g f(x)$, for some $\lambda_g \in \C$. The proof of Leibman's theorem in \cite{MR2877065} follows this approach. From this point, it should be possible to deduce that $P:\Gamma \to X$ is uniquely ergodic, whenever $X$ is a compact nilmanifold and $P$ is a polynomial map (That is, when $P$ is the projection of a polynomial map $P:\Gamma \to G$, where $G$ is a Lie group).
    As mentioned, Leibman has proved that any polynomial map into an abelian group is actually a lift of some polynomial map from a nilpotent quotient, so the extension of Leibman's theorems to locally compact groups is really about extending them to locally compact nilpotent groups. The approach we suggested above does not seem to exploit this reduction.
\subsection{Beyond finitely generated groups}

In the statement of \Cref{thm:Weyl_finitely_generated_to_d_torus}, the domain $\Gamma$ of the polynomial map is assumed to be a finitely generated group amenable group. In fact, using \Cref{thm:poly_to_torus_well_distributed_or_rational} it is easy to prove an even simpler statement when  $\Gamma$ is assumed to be a connected locally compact amenable group. In that case one does not need to worry about finite index subgroups, and $P$ is always well-distributed with respect to Haar measure on a coset of a connected closed subgroup. Using this argument \Cref{thm:Weyl_finitely_generated_to_d_torus} can be generalized to the case where $\Gamma$ is a locally compact amenable group such that $\Gamma/\Gamma_0$ is finitely generated, where $\Gamma_0$ is the connected component of the identity. 

\begin{quest}
    Let $\Gamma$ be a countable amenable group, and  let $P:\Gamma \to (\R/\Z)^d$ be a polynomial map. Is $P$ well-distributed with respect to some probability measure on $(\R/\Z)^d$? 
\end{quest}

In view of  \Cref{thm:Weyl_finitely_generated_to_d_torus}, the question is relevant when $\Gamma$ is not finitely generated. 

The proof of \Cref{thm:finitely_generated_to_compact_poly_uniquely_ergodic} shows that the following question is essentially equivalent:
\begin{quest}
    Let $\Gamma$ be a countable group, and let $G$ be a compact abelian group. Is any  $P \in \Poly_d(\Gamma,G)$  uniquely ergodic?
\end{quest}

For non-discrete groups, $\Gamma$, the orbit closure of a polynomial map $P$ on $\Gamma$ might not be compact and might not support a $\Gamma$-invariant probability measure. The simplest example is the map $P:\R \to \R/\Z$ given by $P(t) = t^2 \mod 1$.
It is easy to check that the orbit of $P$ under $\R$ is given by
\[ \R P = \left\{ t \mapsto t^2 + 2at + a^2 \mod 1:~ a \in \R \right\}.\]
This is a closed but non-compact subset of $C(\R,\R/\Z)$, and it does not support any $\R$-invariant probability measure.
It can be shown that the orbit closure of a polynomial map $P$ from $\R^d$ into a compact group $G$ is compact if and only if the degree of $P$ is at most $1$ (namely if and only if $P$ is an affine homomorphism).
 
The above example also demonstrates the necessity of  compactness assumption on the  orbit closure in \Cref{lem:unique_ergodicity_implies_equidistribution}. 

\subsection{The structure of orbit closures of polynomial maps}

\Cref{prop:homo_uniquely_ergodic} tells us that $\overline{\Gamma P}$ is a coset of a compact subgroup of $C(\Gamma,G)$, for any $P \in \Poly_1(\Gamma,G)$, for any  compact group $G$ and any locally compact group $\Gamma$. 
As the example $P(t)=t^2 \mod 1$ shows, the compactness assumption can fail for polynomials maps of degrees greater than $1$.
When $G$ is compact and $\Gamma$ is a discrete countable group (in particular, when $\Gamma$ is a finitely generated group), $\overline{\Gamma P}$ is always compact, for any $P:\Gamma \to G$.
\Cref{thm:finitely_generated_to_compact_poly_uniquely_ergodic} tells us that in the case that $\Gamma$ is a finitely generated group, $G$ is a compact abelian group and $P:\Gamma \to G$ is a polynomial map, then $\overline{\Gamma P}$ admits a unique $\Gamma$ invariant measure. 

The following result provides additional information about the structure of orbit closures of polynomial maps.

\begin{prop}\label{prop:orbit_closure_finitely_gen_poly}
   Let $\Gamma$ be a finitely generated group,  let $G$ be a topological group, and let $P \in \Poly_d(\Gamma,G)$. Then the orbit closure of $P$ $\overline{ \Gamma P} \subseteq G^\Gamma$ is homeomorphic to $\overline{P^{S^{\le d}}(\Gamma)} \subseteq G^{S^{\le d}}$.
\end{prop}
\begin{proof}
Let $\Gamma$, $S$, and $G$ be as in the statement.
Let $\mathit{Poly}_d(\Gamma,G) \subseteq G^\Gamma$ denote the set of polynomial maps of degree at most $d$ from $\Gamma$ to $G$.
The set $\mathit{Poly}_d(\Gamma,G)$ is closed and $\Gamma$-invariant. Hence for any $P \in \mathit{Poly}_d(\Gamma,G)$ we have 
$\overline{\Gamma P} \subseteq \mathit{Poly}_d(\Gamma,G)$.
\Cref{prop:fg_poly_deteremined_by_finite_set} says that the restriction map  $\mathit{res}_{S^{\le d}}:\mathit{Poly}_d(\Gamma,G) \to G^{S^{\le d}}$ is injective. We will show that the inverse of $\Phi$ is injective. 
The proof of \Cref{prop:fg_poly_deteremined_by_finite_set} reveals that 
for any finitely generated group $\Gamma$ with a finite generating set $S$ and any $d \ge 1$, for any $\gamma \in \Gamma$ there exists $n \in \N$ and $W_\gamma \in \left((S^{\le d}) \cup (S^{-1})^{\le d} \cup \{1_\Gamma\}\right)^n$ such that for any group $G$ and any  polynomial map $P:\Gamma \to G$ of degree at most $d$ $P(\gamma)= \prod_{i=1}^n P(W_\gamma(i))$. 
This shows that the inverse of the function $\mathit{res}_{S^{\le d}}:\mathit{Poly}_d(\Gamma,G) \to G^{S^{\le d}}$ is continuous. Thus restriction of  $\mathit{res}_{S^{\le d}}$ to $\overline{\Gamma P}$ induces a homemorphism between $\overline{ \Gamma P}$ and $\overline{P^{S^{\le d}}(\Gamma)}$.
\end{proof}

In particular, if $G$ is a compact group of finite topological dimension and $\Gamma$ is a finitely generated group, then $\overline{\Gamma P}$ has finite topological dimension for any polynomial map $P:\Gamma \to G$.
Combining  \Cref{prop:orbit_closure_finitely_gen_poly} and \Cref{thm:Weyl_finitely_generated_to_d_torus}, we deduce that the orbit closure of any polynomial map from a finitely generated group into $(\R/\Z)^d$ is a finite union of cosets of a finite-dimensional connected subgroup of $((\R/\Z)^d)^\Gamma$.

If $G$ is a compact abelian group that does not embed in $(\R/\Z)^d$ (namely, if the dual group is not finitely generated), the orbit closure of a polynomial map $P:\Gamma \to G$ need not be a finite union of cosets of a subgroup, even in the basic case $\Gamma = \Z$, as shown by the following example:
Let $G= \prod_p (\Z/ p\Z)$, where the product is over all the odd primes $p$, and let $P:\Z \to G$ be given by
\[P(n)_p = n^2 \mod p \mbox{ for every } n \in \Z \mbox{ and every odd prime } p.\]
Then $\overline{P(\Z)}$ is precisely the subset of $g \in G$ such that $g_p$ is a quadratic residue modulo $p$ for every odd prime $p$. This shows that $\overline{P(\Z)}$ is not a finite union of cosets, so neither is $\overline{\Z P}$.

However, a slight elaboration of the proof of \Cref{thm:Weyl_finitely_generated_to_d_torus} reveals that there exists a (possibly trivial) closed, connected, subgroup $H$ of $G^\Gamma$ such that every connected component of $\overline{\Gamma P}$ is a coset of $H$.


\subsection{Quantitative well-distribution}

In \cite{MR2877065} Green and Tao obtained quantitative equidistribution results for polynomial sequence. Loosely speaking, it was shown that for any polynomial sequence $P:\Z \to X$ into a compact nilmanifold $X$,
for any sufficiently long interval $I$ the empirical measure $\int \delta_{P(k)} dm_I (k)$ is ``close'' to be a mixture of Haar measures on boundedly many cosets of submanifolds. The notion of ``close'' was defined in terms of explicit, effective inequalities, that bound the difference between integrals of the respective measures against sufficiently regular functions. The parameters involved in these inequalities are  the degree of the polynomial, the ``complexity'' of the manifold $X$, and the interval length.
We believe that both \Cref{thm:poly_to_torus_well_distributed_or_rational} and \Cref{thm:Weyl_finitely_generated_to_d_torus} should have ``quantitative'' versions (and so should the conjectured nilpotent extension of \Cref{thm:Weyl_finitely_generated_to_d_torus}). However,
quantitative results of this type are much more cumbersome  to formulate compared to ``qualitative results'', let alone prove.
 Green and Tao applied the  quantitative equidistribution results in their spectacular proof of the  M\"{o}bius and Nilsequences conjecture, where as currently we  do not have in mind a clear application for quantitative versions of \Cref{thm:poly_to_torus_well_distributed_or_rational} or \Cref{thm:Weyl_finitely_generated_to_d_torus}.

\bibliographystyle{amsplain}
\bibliography{library}
\end{document}